\documentclass[british]{article}
\usepackage[latin9]{inputenc}
\usepackage{mathrsfs}
\usepackage{amsthm}
\usepackage{amstext}
\usepackage{amssymb}
\usepackage{graphicx}
\usepackage{esint}

\makeatletter

\newcommand{\binom}[2]{{#1 \choose #2}}

\newcommand{\lyxaddress}[1]{
\par {\raggedright #1
\vspace{1.4em}
\noindent\par}
}
  \theoremstyle{plain}
  \newtheorem{thm}{\protect\theoremname}
 \ifx\proof\undefined\
   \newenvironment{proof}[1][\proofname]{\par
     \normalfont\topsep6\p@\@plus6\p@\relax
     \trivlist
     \itemindent\parindent
     \item[\hskip\labelsep
           \scshape
       #1]\ignorespaces
   }{%
     \endtrivlist\@endpefalse
   }
   \providecommand{\proofname}{Proof}
 \fi
  \theoremstyle{definition}
  \newtheorem{defn}{\protect\definitionname}
  \theoremstyle{plain}
  \newtheorem{cor}{\protect\corollaryname}

\@ifundefined{showcaptionsetup}{}{%
 \PassOptionsToPackage{caption=false}{subfig}}
\usepackage{subfig}
\makeatother

\usepackage{babel}
\providecommand{\corollaryname}{Corollary}
\providecommand{\definitionname}{Definition}
\providecommand{\theoremname}{Theorem}

\begin{document}

\title{Matrix integrals and generating functions for permutations and one-face
rooted hypermaps}

\author{Jacob P Dyer}

\maketitle

\lyxaddress{jpd514@york.ac.uk}

\lyxaddress{Department of Mathematics, University of York, York YO10 5DD, UK}
\begin{abstract}
Closed-form generating functions for counting one-face rooted hypermaps
with a known number of darts by number of vertices and edges is found,
using matrix integral expressions relating to the reduced density
operator of a bipartite quantum system. A recursion relation for these
generating functions is also found. The method for computing similar
generating functions for two-face rooted hypermaps by number of vertices
and edges is outlined.\end{abstract}
\begin{description}
\item [{Keywords:}] enumeration, rooted hypermap, bipartite quantum system,
matrix integral, generating function
\end{description}

\section{Introduction}

The link between enumerative combinatorics and matrix integration
has been discussed for some time \cite{Zvonkin1997,Lando2004}. In
the physics community it has seen extensive use in the approximation
of quantities in quantum field theory, with the problem of evaluating
matrix integrals being replaced by that of evaluating summations over
graphs and other combinatoric objects \cite{Brezin1978,Itzykson1980}.
A familiar related concept is that of Feynman diagrams, where a path
integral is expressed as a sum of terms expressed using graphs with
particular properties, allowing the path integral to be approximated
as a sum over a finite number of graphs.

It can be used in the opposite sense as well; just as counting graphs
may be used to evaluate intractable integrals, in some cases it is
possible to evaluate the matrix integrals exactly in order to find
generating functions which can be used to count graphs. This has gained
relatively little attention, although examples are known, such as
counting one-face maps with given genus and number of edges \cite{Bessis1980,Harer1986,Zvonkin1997}.

However, the fact that matrix integration is potentially able to produce
exact generating functions for combinatoric classes makes this a very
powerful tool; generating functions have many uses in enumerative
combinatorics, as basic operations on generating functions such as
addition and multiplication have combinatorial meanings in terms of
unions and products of sets \cite[pp 3-4]{Stanley1997}.

In this paper we further demonstrate the power of matrix integrals
by deriving a closed-form generating function for counting rooted
hypermaps with one face and $r$ darts (which are defined in Section
\ref{sub:Hypermaps}) by number of edges and vertices. This is done
using matrix integral expressions arising in the study of finite-dimensional
bipartite quantum systems (which are defined in Section \ref{sub:Bipartite-quantum-systems})

Considerable work has been done on the enumeration of hypermaps prior
to this. Explicit expressions for the number of hypermaps with a given
number of darts have been found for genus zero (planar) \cite{Walsh1975}
and genus one (toroidal) \cite{Arques1987} hypermaps, and recursive
formulae have then been used to extend this to higher genus \cite{Mednykh2010}.
Other methods have also been used, including iterating explicitly
through all hypermaps of a given size and computing their properties;
this has in particular been used to count rooted hypermaps for any
number of vertices, edges, faces and up to twelve darts \cite{Walsh2012}.

In comparison to this past work, the method presented here is neither
limited to a single genus, nor reliant on recursion (the generating
function for hypermaps with $r$ darts is a closed-form function of
$r$ and not dependent on the generating functions for lower $r$)
or direct enumeration. It is also, as far as we are aware, the first
time that matrix integrals have been used to derive generating functions
for hypermaps.

Having said this, we will also show that there is a recursive solution
to this problem, which can find the generating function for $r$ darts
in terms of those for $r-1$ and $r-2$ darts using a simple recursion
relation (see Section \ref{sub:Recursion-relation}). The recursive
method can compute generating functions even more efficiently if the
aim is to find all such functions up to a particular order, although
the ability to compute a single generating function without computing
the preceeding ones is lost.

The one-face rooted hypermap case is in fact a demonstration of a
much more general property of certain matrix integrals (specifically
polynomial integrals over the unit sphere), relating them to sums
over the symmetric group of permutations. We will demonstrate the
broader link using the simpler case of enumeration of permutations
by number of cycles as an example, again deriving the result in the
form of a generating function (see Section \ref{sub:Permutations}).
We will also show our method's broader applicability to rooted hypermaps
in particular by discussing how it may be used to enumerate rooted
hypermaps with two faces as well (again by number of darts, edges
and vertices; see Section \ref{sub:Beyond-one-face}).

\section{Quantum systems and hypermaps}

\subsection{Bipartite quantum systems\label{sub:Bipartite-quantum-systems}}

See \cite[pp 180-186]{Sakurai2013} for an in-depth introduction to
density operators.

A \emph{finite-dimensional bipartite quantum system} $S$ is a system
with an associated Hilbert space $\mathscr{H}\equiv\mathbb{C}^{mn}$
for positive integers $m$ and $n$, such that it can be decomposed
into the union of two systems $A$ and $B$ with Hilbert spaces $\mathscr{H}_{A}\equiv\mathbb{C}^{m}$
and $\mathscr{H}_{B}\equiv\mathbb{C}^{n}$ respectively, such that
$\mathscr{H}=\mathscr{H}_{A}\otimes\mathscr{H}_{B}$. A state of this
system is described by a \emph{density operator} $\hat{\rho}$, which
is a positive Hermitian operator with unit trace acting on $\mathscr{H}$.
We will only be considering \emph{pure states} of $S$, which are
defined as those states where $\hat{\rho}$ is a projection operator.

In this paper we are interested in the \emph{reduced density operator}
for the subsystem $A$, denoted $\hat{\rho}^{A}$, which is found
by taking the partial trace of $\hat{\rho}$ over $\mathscr{H}_{B}$.
$\hat{\rho}^{A}$ is not necessarily in a pure state even when $\hat{\rho}$
is pure, meaning that $\text{Tr}[(\hat{\rho}^{A})^{r}]\ne1$. We wish
to average this expression over all possible pure states of $S$,
and will do so by parametrising $\hat{\rho}$ in terms of a unit vector
$|z\rangle$ in $\mathscr{H}$ (i.e. such that $\hat{\rho}=|z\rangle\langle z|$)
and integrating with respect to the invariant measure on the unit
sphere $S^{2mn-1}$ containing $|z\rangle$ (this measure has been
used for this purpose in other papers, such as \cite{Lubkin1978,Page1993}
-- in \cite{Lubkin1978} it is referred to as the ``Haar measure'').

We express the mean $\langle\text{Tr}[(\hat{\rho}^{A})^{r}]\rangle$
in terms of the function $P_{r}(m,n)$, defined in Theorem \ref{thm:matrix-integral};
this function is important as we will prove in Section \ref{sub:Hypermaps}
that it is equivalent to the generating function for enumerating one-face
rooted hypermaps with $r$ darts, as we will prove in Theorem \ref{thm:Hypermap}
in Section \ref{sub:Hypermaps}.

\bigskip{}

\begin{thm}
\label{thm:matrix-integral}If $\hat{\rho}^{A}$ is the reduced density
operator of an $m$-dimensional subsystem of an $mn$-dimensional
bipartite quantum system $S$, then for any integer $r$,
\begin{equation}
\langle\text{Tr}[\hat{\rho}_{A}^{r}]\rangle=\frac{\Gamma(mn)}{\Gamma(mn+r)}P_{r}(m,n).\label{eq:quantum-ident}
\end{equation}
Here $P_{r}(m,n)$ is defined as 
\begin{equation}
P_{r}(m,n)=\left.\frac{\partial}{\partial\alpha_{a_{1}b_{1}}}\frac{\partial}{\partial\alpha_{a_{2}b_{2}}}\ldots\frac{\partial}{\partial\alpha_{a_{r}b_{r}}}(\alpha_{a_{1}b_{2}}\ldots\alpha_{a_{r-1}b_{r}}\alpha_{a_{r}b_{1}})\right|_{\alpha=0},\label{eq:P_r}
\end{equation}
where $\alpha_{ab}$ are the components of an $m\times n$ real matrix
$\alpha$.\end{thm}
\begin{proof}
Let $\{|\phi_{1}^{A}\rangle,\ldots,|\phi_{m}^{A}\rangle\}$ and $\{|\phi_{1}^{B}\rangle,\ldots,|\phi_{n}^{B}\rangle\}$
be two orthonormal basis sets which span $\mathscr{H}_{A}$ and $\mathscr{H}_{B}$
respectively. In terms of this basis we parametrise a general vector
$|z\rangle\in\mathscr{H}$ as
\[
|z\rangle=z_{ab}|\phi_{a}^{A}\rangle\otimes|\phi_{b}^{B}\rangle,
\]
where $z$ is an $m\times n$ complex matrix, and the convention of
summing over repeated indices is assumed; in this and all future working,
any index represented by the letter $a$ runs from $1$ to $m$ and
any index represented by the letter $b$ runs from $1$ to $n$. In
terms of the components of $z$, the inner product of two such vectors
is 
\[
\langle x|z\rangle=x_{ab}^{*}z_{ab},
\]
so the condition for $|z\rangle$ to be a unit vector is 
\[
\langle z|z\rangle=z_{ab}^{*}z_{ab}=1.
\]

We now express $\hat{\rho}^{A}$ using the \emph{reduced density matrix}
$\rho^{A}$, which in our chosen basis has components
\[
\rho_{a_{1}a_{2}}^{A}=\langle\phi_{a_{1}}^{A}|\hat{\rho}^{A}|\phi_{a_{2}}^{A}\rangle=z_{a_{1}b}z_{a_{2}b}^{*},
\]
and use these components to construct the following matrix-integral
expansion for $\langle\text{Tr}[(\hat{\rho}^{A})^{r}]\rangle$:
\begin{equation}
\langle\text{Tr}[(\hat{\rho}^{A})^{r}]\rangle=\frac{1}{S_{2mn-1}}\int_{S^{2mn-1}}d\sigma z_{a_{1}b_{1}}z_{a_{2}b_{1}}^{*}z_{a_{2}b_{2}}z_{a_{3}b_{2}}^{*}\ldots z_{a_{r}b_{r}}z_{a_{1}b_{r}}^{*},\label{eq:matrix-integral}
\end{equation}
where $d\sigma$ is the invariant volume element on the hypersphere
$S^{2mn-1}$ and 
\[
S_{2mn-1}=\frac{2\pi^{mn}}{\Gamma(mn)}
\]
is the volume of $S^{2mn-1}$.

(\ref{eq:matrix-integral}) is a polynomial in the components of $z$
and $z^{*}$ containing only terms of order $2r$. It is known that
such an integral performed over the unit sphere can be converted into
a Gaussian integral \cite{Folland2001}. Following this procedure,
we multiply (\ref{eq:matrix-integral}) by the expression
\[
\frac{2}{\Gamma(mn+r)}\int_{0}^{\infty}\lambda^{2mn+2r-1}e^{-\lambda^{2}}d\lambda,
\]
which equals unity by the definition of the gamma function. Then we
define $x=\lambda z$, and note that $x_{ab}^{*}x_{ab}=\lambda^{2}$
and $d^{2mn}x=\lambda^{2mn-1}d\lambda d\sigma$. Using this change
of variables, we get that
\begin{equation}
\langle\text{Tr}[(\hat{\rho}^{A})^{r}]\rangle=\frac{\Gamma(mn)\pi^{-mn}}{\Gamma(mn+r)}\int_{\mathbb{C}^{mn}}d^{2mn}xe^{-x_{ab}^{*}x_{ab}}x_{a_{1}b_{1}}x_{a_{2}b_{1}}^{*}\ldots x_{a_{r}b_{r}}x_{a_{1}b_{r}}^{*}.\label{eq:x-integral}
\end{equation}
Next we define a pair of real $m\times n$ matrices $\alpha$ and
$\beta$. We use these to remove the polynomial part from (\ref{eq:x-integral}):
\begin{eqnarray*}
\langle\text{Tr}[(\hat{\rho}^{A})^{r}]\rangle & = & \frac{\Gamma(mn)\pi^{-mn}}{\Gamma(mn+r)}\frac{\partial}{\partial\alpha_{a_{1}b_{1}}}\frac{\partial}{\partial\beta_{a_{2}b_{1}}}\ldots\frac{\partial}{\partial\alpha_{a_{r}b_{r}}}\frac{\partial}{\partial\beta_{a_{1}b_{r}}}\\
 &  & \quad\left.\int_{\mathbb{C}^{mn}}d^{2mn}xe^{-x_{ab}^{*}x_{ab}+\alpha_{ab}x_{ab}+x_{ab}^{*}\beta_{ab}}\right|_{\alpha,\beta=0}\\
 & = & \frac{\Gamma(mn)}{\Gamma(mn+r)}\left.\frac{\partial}{\partial\alpha_{a_{1}b_{1}}}\frac{\partial}{\partial\beta_{a_{2}b_{1}}}\ldots\frac{\partial}{\partial\alpha_{a_{r}b_{r}}}\frac{\partial}{\partial\beta_{a_{1}b_{r}}}e^{\alpha_{ab}\beta_{ab}}\right|_{\alpha,\beta=0}.
\end{eqnarray*}
As the derivatives all commute with each other, we can choose what
order to perform them in. Choosing to perform the $\beta$-derivatives
first, we get
\begin{eqnarray*}
\langle\text{Tr}[(\hat{\rho}^{A})^{r}]\rangle & = & \frac{\Gamma(mn)}{\Gamma(mn+r)}\left.\frac{\partial}{\partial\alpha_{a_{1}b_{1}}}\ldots\frac{\partial}{\partial\alpha_{a_{r}b_{r}}}(\alpha_{a_{1}b_{2}}\ldots\alpha_{a_{r}b_{1}})\right|_{\alpha=0}\\
 & \equiv & \frac{\Gamma(mn)}{\Gamma(mn+r)}P_{r}(m,n).
\end{eqnarray*}

\end{proof}

\subsection{Permutations\label{sub:Permutations}}

In Theorem \ref{thm:matrix-integral} we have shown that the matrix
integral (\ref{eq:matrix-integral}) (which has the form of a polynomial
function of matrix components integrated over the unit sphere) can
equivalently be expressed as the multiderivative expression (\ref{eq:P_r}).
Evaluation of $P_{r}(m,n)$ based on (\ref{eq:P_r}) is directly linked
to the symmetric group of permutations in that it involves summing
over all possible pairings of $\frac{\partial}{\partial\alpha}$ terms
with $\alpha$ terms; there are $r!$ such pairings, corresponding
to the different permutations on $r$ elements. In this section we
will demonstrate this fact more clearly by using a related matrix
integral to enumerate the set of permutations on $r$ elements by
number of cycles.

Before we begin, it should be noted that the result in this case is
alreadly known; the number of permutations on the set $\{1\ldots r\}$
with $k$ cycles is the unsigned Stirling number of the first kind
$\mathfrak{s}(r,k)$ \cite[p 234]{Comtet1994}, and these have a generating
function \cite[p 213]{Comtet1994}
\begin{equation}
\sum_{k=1}^{r}\mathfrak{s}(r,k)m^{k}=\frac{\Gamma(m+r)}{\Gamma(m)}.\label{eq:Stirling-gf}
\end{equation}
Our aim here is to show how the matrix integrals are used to find
generating functions, however, so we will in this section derive the
generating function (\ref{eq:Stirling-gf}) directly without considering
the properties of the counting functions themselves (denoted $c_{r}(k)$
from now on to make it clear that nothing of their value is being
assumed).

\bigskip{}

\begin{thm}
Let $c_{r}(k)$ be the number of permutations on the set $R=\{1\ldots r\}$
with a decomposition into $k$ disjoint cycles. Then
\[
\sum_{k=1}^{r}c_{r}(k)m^{k}=\frac{\Gamma(m+r)}{\Gamma(m)}
\]
for any $r\ge1$.\end{thm}
\begin{proof}
Define 
\begin{equation}
P_{r}(m)=\frac{2}{\Gamma(m)}\int_{0}^{\infty}x^{2m+2r-1}e^{-x^{2}}dx.\label{eq:P_r(m)}
\end{equation}
Using the definition of the gamma function, we know that 
\begin{equation}
P_{r}(m)=\frac{\Gamma(m+r)}{\Gamma(m)}.\label{eq:P_r(m)-eval}
\end{equation}

Let $u$ be a random vector in $\mathbb{C}^{m}$, where $x=\sqrt{u\cdot u}$.
Noting that $d^{2m}u=x^{2m-1}dxd\sigma$ (where $d\sigma$ is the
invariant volume element on the sphere $S^{2m-1}$), we express (\ref{eq:P_r(m)})
as
\begin{eqnarray*}
P_{r}(m) & = & \frac{2}{\Gamma(m)}\int_{0}^{\infty}(u\cdot u)^{r}e^{-u\cdot u}x^{2m-1}dx\\
 & = & \frac{2}{\Gamma(m)}\frac{\int_{\mathbb{C}^{m}}(u\cdot u)^{r}e^{-u\cdot u}d^{2m}u}{\int_{S^{2m-1}}d\sigma}.
\end{eqnarray*}
$\int_{S^{2m-1}}d\sigma=2\pi^{m}/\Gamma(m)$ is the volume of the
sphere $S^{2m-1}$, so
\[
P_{r}(m)=\pi^{-m}\int_{\mathbb{C}^{m}}(u\cdot u)^{r}e^{-u\cdot u}d^{2m}u.
\]
We then evaluate the integral by defining two real-valued $m$-vectors
$\alpha$ and $\beta$, and writing
\begin{eqnarray*}
P_{r}(m) & = & \pi^{-m}\left.\left(\frac{\partial}{\partial\alpha_{i}}\frac{\partial}{\partial\beta_{i}}\right)^{r}\int_{\mathbb{C}^{m}}e^{-u\cdot u+\alpha\cdot u+u\cdot\beta}d^{2m}u\right|_{\alpha,\beta=0}\\
 & = & \left.\left(\frac{\partial}{\partial\alpha_{i}}\frac{\partial}{\partial\beta_{i}}\right)^{r}e^{\alpha\cdot\beta}\right|_{\alpha,\beta=0},
\end{eqnarray*}
where the convention of summation over the repeated index $i$ (running
from $1$ to $m$) is assumed. The derivatives all commute, so we
can choose to perform all of the $\beta$-derivatives first. After
doing so we get
\begin{equation}
P_{r}(m)=\left.\frac{\partial}{\partial\alpha_{i_{1}}}\ldots\frac{\partial}{\partial\alpha_{i_{r}}}(\alpha_{i_{1}}\ldots\alpha_{i_{r}})\right|_{\alpha=0}.\label{eq:permutation-func}
\end{equation}

To fully expand the derivatives in this expression out, we need to
account for every possible pairing of the $r$ derivative terms with
the $r$ $\alpha_{i}$ terms, leading to a summation over all permutations
$\sigma$ on $R$ (i.e. all permutations in the symmetric group $Sym_{r}$).
Therefore,
\[
P_{r}(m)=\sum_{\sigma\in Sym_{r}}\prod_{s=1}^{r}\frac{\partial}{\partial\alpha_{i_{s}}}\alpha_{i_{\sigma(s)}}=\sum_{\sigma\in Sym_{r}}\prod_{s=1}^{r}\delta[i_{s},i_{\sigma(s)}]
\]
where $\delta[\ldots,\ldots]$ is the Kronecker delta.

In the term for a given $\sigma$ every $i$ index in the product
is implicitly contracted over, and any cycle in $\sigma$ gives rise
to a single self-contained contraction. e.g. a cycle $(134)$ would
result in a term looking like
\[
\delta[i_{1},i_{3}]\delta[i_{3},i_{4}]\delta[i_{4},i_{1}]=\delta[i_{1},i_{1}]=m.
\]
 Each cycle therefore contributes a factor of $m$, such that the
term corresponding to $\sigma$ is equal to $m^{a}$, where $a$ is
the number of cycles in $\sigma$. This means that
\[
P_{r}(m)=\sum_{k=1}^{r}c_{r}(k)m^{k}
\]
where $c_{r}(k)$ is the number of permutations on $R$ with $k$
cycles. (\ref{eq:P_r(m)-eval}) therefore gives
\[
\sum_{k=1}^{r}c_{r}(k)m^{k}=\frac{\Gamma(m+r)}{\Gamma(m)}.
\]

\end{proof}
\bigskip{}

This proof demonstrates that the procedure for evaluating mutiderivative
expressions such as (\ref{eq:permutation-func}) involves a summation
over permutations, and that the resulting function acts as a generating
function counting the permutations by number of cycles. In the next
section we will explain how rooted hypermaps can be expressed using
permutations, and then evaluate (\ref{eq:P_r}) using the above method
to show that it is the generating function for rooted one-face hypermaps
with $r$ darts.

This generating function is also directly related to $P_{r}(m,n)$
through the identity
\[
P_{r}(m)=P_{r}(m,1).
\]
We will state the reason for this at the end of the next section.

\subsection{Hypermaps\label{sub:Hypermaps}}

The following defintions come from \cite{Lando2004}.

A \emph{hypergraph} $G$ is a pair $\{V,E\}$ consisting of a (non-empty)
set $V$ of vertices and a family $E$ of edges, each edge being a
non-empty family of vertices from $V$. This is a generalisation of
the standard notion of a graph; whereas in graphs each edge must contain
exactly two vertices, in hypergraphs they can contain any positive
number of vertices. An edge in a hypermap can contain any given vertex
more than once, and each edge-vertex connection is called a \emph{dart}.

A \emph{map} is an embedding of a graph onto an orientable surface
such that the graph's complement consists of regions (called faces)
which are isomorphic to the open unit disc; such an embedding is called
a \emph{2-cell embedding}. Through this embedding, a map gains a \emph{genus}
which is equal to the genus of the surface it's embedded in. There
is a bijection between hypergraphs and bicoloured bipartite graphs
\cite{Walsh1975}, so the concept of maps generalises naturally to
hypergraphs, such that a \emph{hypermap} can be defined as an 2-cell
embedding of a hypergraph on an orientable surface. The definitions
of faces and genus carry over.

In addition, a \emph{rooted hypermap} is a hypermap where one dart
has been labelled as a ``root'', i.e. an isomorphism between rooted
hypermaps must preserve the identity of the root as well as the hypermap's
connectivity.

There is another definition of a hypermap (also from \cite{Lando2004})
which will be of more use here. It is equivalent to the above, but
defines the embedding in a combinatorial manner without having to
consider actual surfaces. This defines hypermaps in terms of \emph{3-constellations}
- ordered triples $\{\xi,\chi,\eta\}$ of permutations on $r$ elements
where (a) the group generated by $\{\xi,\chi,\eta\}$ acts transitively
on the set $\{1,\ldots,r\}$ and (b) the product $\xi\chi\eta$ equals
the identity.

\bigskip{}

\begin{defn}
A hypermap $H$ with $r$ darts is a 3-constellation $\{\xi,\chi,\eta\}$
of permutations acting on the set $\{1,\ldots,r\}$, which represents
the hypermap's darts. The three permutations correspond to cycling
the darts adjacent to any given face, edge or vertex%
\footnote{In the initial description of hypermaps given previously, which edge
and vertex each dart is `adjacent to' is clear, as darts were defined
as connections between edges and vertices. If a hypermap is embedded
on an orientable surface, you can uniquely define the ``adjacent
face'' for a dart by moving anticlockwise from the dart around its
adjacent vertex. While each dart borders two faces, it is only adjacent
to one in this sense. Every other dart bordering a given face is adjacent
to it.%
} respectively in an aticlockwise direction around their adjacent object.

\bigskip{}

\begin{figure}[h]
\hfill{}\includegraphics[width=6cm]{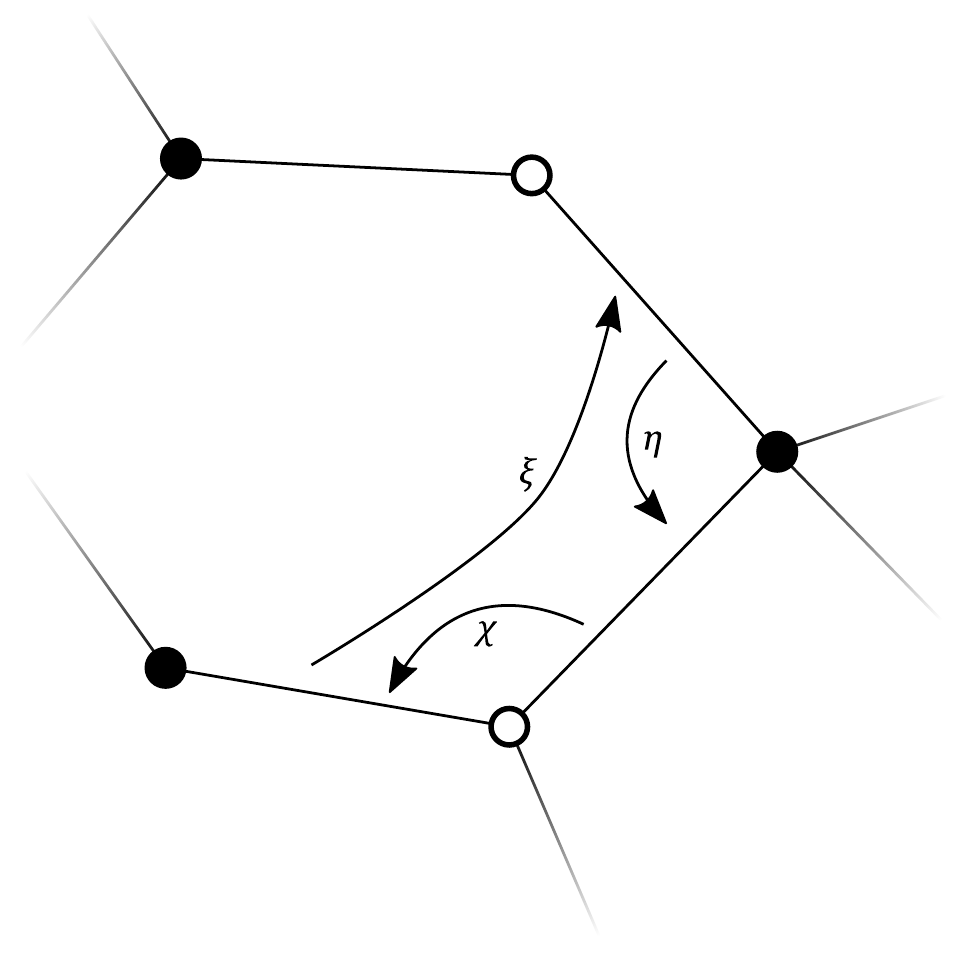}\hfill{}

\protect\caption{\label{fig:perm-action} Dart mappings within a hypermap with 3-constellation
$\{\xi,\chi,\eta\}$. The permutation $\eta$ rotates each dart one
place anti-clockwise around a vertex (black circle). $\chi$ does
the same around an edge (white cicrle). $\xi$ rotates the dart anticlockwise
around a face to the next dart adjacent to that face. The three mappings
arranged like this show why $\xi\chi\eta=1$.}

\end{figure}

Figure \ref{fig:perm-action} shows how the three permutations act
on the darts in a hypermap. It is clear that if e.g. $\eta$ is applied
a number of times equalling the order of a given vertex $V$, then
all darts adjacent to $V$ will map to themselves. Each vertex in
a hypermap is therefore associated with a cycle in the permutation
$\eta$. Similarly, the cycles in $\xi$ and $\chi$ are associated
with faces and edges respectively. We have already shown in Section
\ref{sub:Permutations} that matrix integration provides a way to
enumerate permutations in terms of the number of cycles, so the correspondence
between permutation cycles and edges/vertices in hypermaps now allows
us to enumerate one-face rooted hypermaps by number of edges and vertices.
This will make use of the quantum matrix integrals encountered in
Section \ref{sub:Bipartite-quantum-systems}.

\bigskip{}
\end{defn}
\begin{thm}
\label{thm:Hypermap}If $\hat{\rho}^{A}$ is the reduced density operator
of an $m$-dimensional subsystem of an $mn$-dimensional bipartite
quantum system $S$, then for any integer $r$

\[
\langle\text{Tr}[(\hat{\rho}^{A})^{r}]\rangle=\frac{\Gamma(mn)}{\Gamma(mn+r)}\sum_{e,v}h_{r}^{(1)}(e,v)m^{e}n^{v},
\]
where $h_{r}^{(1)}(e,v)$ is the number of rooted hypermaps with one
face, $r$ darts, $v$ vertices and $e$ edges, and the summation
is over all possible values of these parameters for given $r$.\end{thm}
\begin{proof}
Using Theorem \ref{thm:matrix-integral}, we restate this as
\[
P_{r}(m,n)=\sum_{e,v}h_{r}^{(1)}(e,v)m^{e}n^{v}.
\]
We will prove this relation by considering how to evaluate $P_{r}$
as a polynomial in $m$ and $n$ . 

\begin{figure}[h]
\hfill{}\includegraphics[width=10cm]{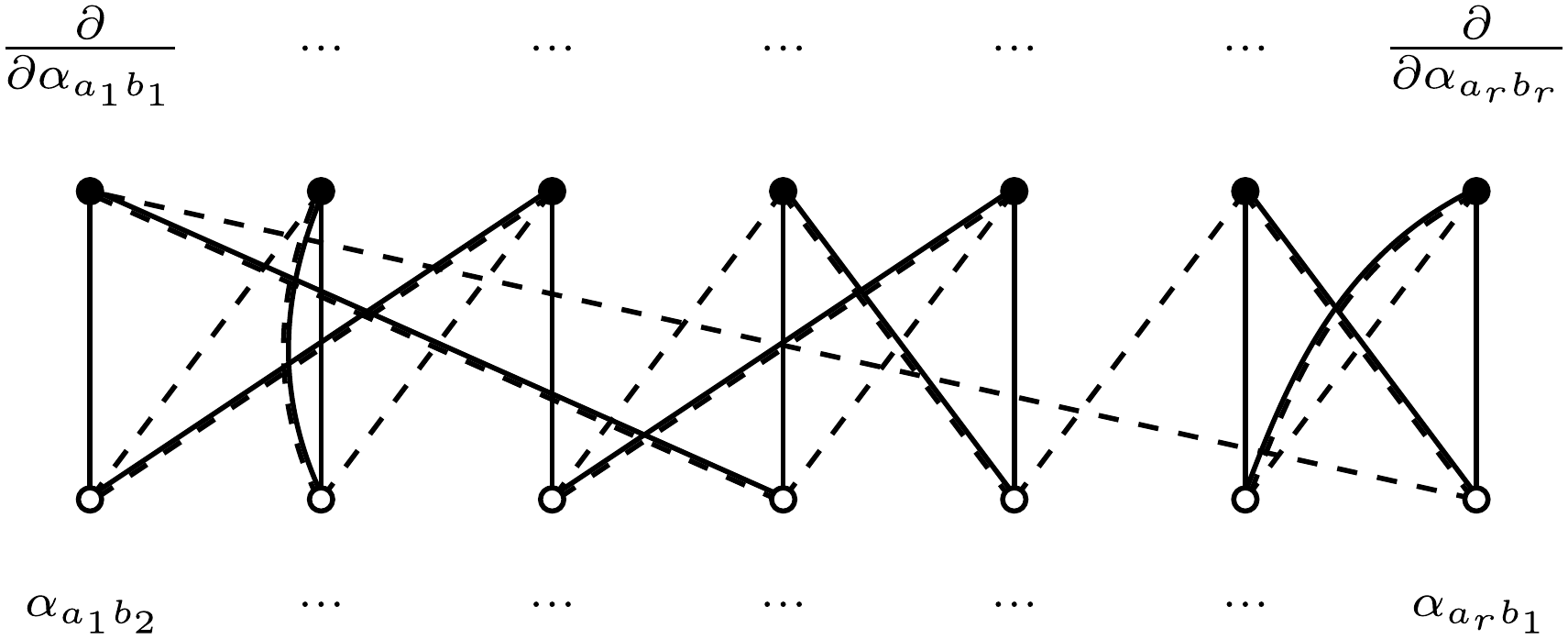}\hfill{}

\protect\caption{\label{fig:contractions}A diagram showing all contractions involved
in evaluating a term in $P_{r}(m,n)$ corresponding to a random permutation
$\sigma$ (in this case $r=7$ and the permutation $\sigma$ this
diagram corresponds to is $(1453)(2)(67)$). The vertices correspond
to the $\frac{\partial}{\partial\alpha}$ and $\alpha$ terms as labelled.
The single solid and dotted lines represent contractions caused by
shared $a$ and $b$ indices respectively. The double solid/dotted
lines are the contractions caused by the pairings of $\frac{\partial}{\partial\alpha}$
and $\alpha$ terms, as given by $\sigma$. The term corresponding
to this diagram is $m^{3}n^{3}$, the exponents of $m$ and $n$ found
by counting closed loops of solid and dotted lines respectively.}
\end{figure}

As with $P_{r}(m)$ in \ref{sub:Permutations}, $P_{r}(m,n)$ as given
in (\ref{eq:P_r}) is a multiderivative function which can be expanded
as a sum over permutations, each permutation corresponding to different
pairing of $\frac{\partial}{\partial\alpha}$ terms and $\alpha$
terms. Defining the cycle permutation 
\begin{equation}
\xi=(123\ldots r),\label{eq:cycle}
\end{equation}
we expand $P_{r}$ as

\begin{equation}
P_{r}(m,n)=\sum_{\sigma\in Sym_{r}}\prod_{i=1}^{r}\frac{\partial}{\partial\alpha_{a_{i}b_{i}}}\alpha_{a_{\sigma(i)}b_{\xi\sigma(i)}}=\sum_{\sigma\in Sym_{r}}\prod_{i=1}^{r}\delta[a_{i},\sigma(a_{i})]\delta[b_{i},\xi\sigma(b_{i})],\label{eq:P_r-enum}
\end{equation}
where $\delta[\ldots,\ldots]$ represents the Kronecker delta. The
term corresponding to a given $\sigma$ therefore consists of a completely
contracted product of Kronecker deltas, some with dimension $m$ (those
with $a$ indices) and some with dimension $n$ (those with $b$ indices).
When these are contracted, they produce terms of the form $m^{A}n^{B}$
where $A$ is the number of cycles in $\sigma$ and $B$ the number
of cycles in $\xi\sigma$ (each cycle creates a completely closed
loop of contractions such as $\delta[a_{1},a_{2}]\delta[a_{2},a_{4}]\delta[a_{4},a_{1}]=m$).
See Figure \ref{fig:contractions} for a graphical representation
of these cycles.

Therefore, the term $m^{A}n^{B}$ associated with $\sigma$ corresponds
to a hypermap $H_{\sigma}$ with 3-constellation $\{\xi,\chi,\eta\}=\{\xi,\sigma,(\xi\sigma)^{-1}\}$.
$A$ is the number of cycles in $\chi$, which is the number of edges
in the hypermap, and $B$ is the number of cycles in $\eta$, or the
number of vertices in the hypermap. As $\xi$ has only one cycle,
these hypermaps necessarily have exactly one face.

This correspondence between permutations $\sigma$ and hypermaps is
not bijective, as two hypermaps with 3-constellations $\{\xi,\chi,\eta\}$
and $\{\xi',\chi',\eta'\}$ related by 
\begin{equation}
\xi=\tau\xi'\tau^{-1},\,\chi=\tau\chi'\tau^{-1},\,\eta=\tau\eta'\tau^{-1}\label{eq:tau-transform}
\end{equation}
for some permutation $\tau$ are isomorphic to each other (the transformation
(\ref{eq:tau-transform}) amounts to a simple reordering of the darts
in the hypermap without changing its connectivity). Given that we
have specified $\xi$ exactly in (\ref{eq:cycle}), however, the only
isomorphisms which preserve $\xi$ are those where $\tau$ is an integer
power of $\xi$. To make the correspondence bijective we simply need
to remove this cyclic equivalence of darts, and we do so by labelling
one dart in the hypermap as the ``root'' to make it distinct from
the others. The summation over permutations in (\ref{eq:P_r-enum})
is therefore equivalent to a summation over all nonisomorphic rooted
hypermaps.

This finally leads us to the conclusion that
\[
P_{r}(m,n)=\sum_{e,v}h_{r}^{(1)}(e,v)m^{e}n^{v},
\]
 and 
\[
\langle\text{Tr}[(\hat{\rho}^{A})^{r}]\rangle=\frac{\Gamma(mn)}{\Gamma(mn+r)}\sum_{e,v}h_{r}^{(1)}(e,v)m^{e}n^{v}.
\]

\end{proof}
\bigskip{}

This proof quickly leads to the additional result:

\bigskip{}

\begin{cor}
\label{cor:all-count}There are $r!$ rooted hypermaps with one face
and $r$ darts.\end{cor}
\begin{proof}
This follows from the bijection proved in Theorem \ref{thm:Hypermap}
between rooted hypermaps with one face and $r$ darts and permutations
in $Sym_{r}$. There are $r!$ permutations in $Sym_{r}$, so there
are $r!$ such hypermaps.
\end{proof}
\bigskip{}

The bijection also gives a direct link back to the permutation generating
function $P_{r}(m)$ given in Section \ref{sub:Permutations}. As
each one-face rooted hypermap corresponds bijectively to the permutation
$\chi$ in its 3-constellation representation, and the number of cycles
in $\chi$ gives the number of edges in the hypermap, then summing
over all one-face rooted hypermaps with $r$ darts and $k$ edges
gives the number of permutations on $r$ elements with $k$ cycles.
We can equivalently express this as

\begin{equation}
P_{r}(m)=P_{r}(m,1).\label{eq:P_r-identity}
\end{equation}

\section{Computing generating functions}

Theorem \ref{thm:Hypermap} immediately suggests a procedure for calculating
$P_{r}(m,n)$ i.e. by performing the summation over permutations given
in (\ref{eq:P_r-enum}). This can be done in approximately $\mathcal{O}(r\cdot r!)$
time, as there are $r!$ permutations to sum over and it would take
$\mathcal{O}(r)$ time to count the number of cycles for each one.
However, this algorithm is equivalent to just enumerating the individual
hypermaps, and it would be preferable to instead find a way of expressing
the generating functions in a closed form as functions of $m$, $n$
and $r$.

Such a method does exist, and it arises from the fact that
\[
P_{r}(m,n)=\frac{\Gamma(mn+r)}{\Gamma(mn)}\langle\text{Tr}[(\hat{\rho}^{A})^{r}]\rangle.
\]
While Theorem \ref{thm:matrix-integral} is proven by evaluating this
mean as a Gaussian integral, other methods for integrating such expressions
exist. We used one such method in a recent paper, and derived the
closed-form expression \cite{Dyer2014}
\begin{equation}
\langle\text{Tr}[(\hat{\rho}^{A})^{r}]\rangle=\frac{\Gamma(mn)}{r\Gamma(mn+r)}\sum_{k=0}^{m-1}\frac{(-1)^{k}\Gamma(m+r-k)\Gamma(n+r-k)}{k!\Gamma(r-k)\Gamma(m-k)\Gamma(n-k)}\label{eq:P_r-closed}
\end{equation}
for integers $m$ and $n$ and real $r$. We can then rearrange this
expression into a closed-form generating function as follows:

\bigskip{}

\begin{thm}
\label{thm:closed-form}
\begin{equation}
P_{r}(m,n)=\frac{1}{r!}\sum_{k=0}^{r-1}(-1)^{k}\binom{r-1}{k}\frac{\Gamma(m+r-k)}{\Gamma(m-k)}\frac{\Gamma(n+r-k)}{\Gamma(n-k)}\label{eq:gf}
\end{equation}
for any positive integers $r$, $m$ and $n$.\end{thm}
\begin{proof}
From (\ref{eq:quantum-ident}) and (\ref{eq:P_r-closed}), we get
that
\[
P_{r}(m,n)=\frac{1}{r}\sum_{k=0}^{m-1}\frac{(-1)^{k}\Gamma(m+r-k)\Gamma(n+r-k)}{k!\Gamma(r-k)\Gamma(m-k)\Gamma(n-k)}
\]
for integers $m$ and $n$ and real $r$. As we are only concerned
with integers $r$ in this paper, however, we can simplify it to
\begin{eqnarray*}
P_{r}(m,n) & = & \frac{1}{r}\sum_{k=0}^{m-1}\frac{(-1)^{k}\Gamma(m+r-k)\Gamma(n+r-k)}{k!(r-k-1)!\Gamma(m-k)\Gamma(n-k)}\\
 & = & \frac{1}{r!}\sum_{k=0}^{m-1}(-1)^{k}\binom{r-1}{k}\frac{\Gamma(m+r-k)}{\Gamma(m-k)}\frac{\Gamma(n+r-k)}{\Gamma(n-k)}.
\end{eqnarray*}
The term inside the summation is zero if $k>m$ or $k>r$, so we are
free to change the upper bound of the summation to any integer greater
than or equal to $[\min(m,r)-1]$. Therefore,
\[
P_{r}(m,n)=\frac{1}{r!}\sum_{k=0}^{r-1}(-1)^{k}\binom{r-1}{k}\frac{\Gamma(m+r-k)}{\Gamma(m-k)}\frac{\Gamma(n+r-k)}{\Gamma(n-k)}.
\]

\end{proof}
\bigskip{}

This expression is, for known $r$, a sum over a fixed finite number
of polynomials in $m$ and $n$, making it suitable for use as a generating
function. It is also clearly symmetric in $m$ and $n$, which reflects
a fundamental property of rooted hypermaps -- that there is a bijection
from the set of rooted hypermaps to itself which consists of replacing
each edge with a vertex and vice versa.

\subsection{Computation}

\begin{figure}[h]
\hfill{}\subfloat[\label{fig:graph-a}]{\includegraphics[width=4cm]{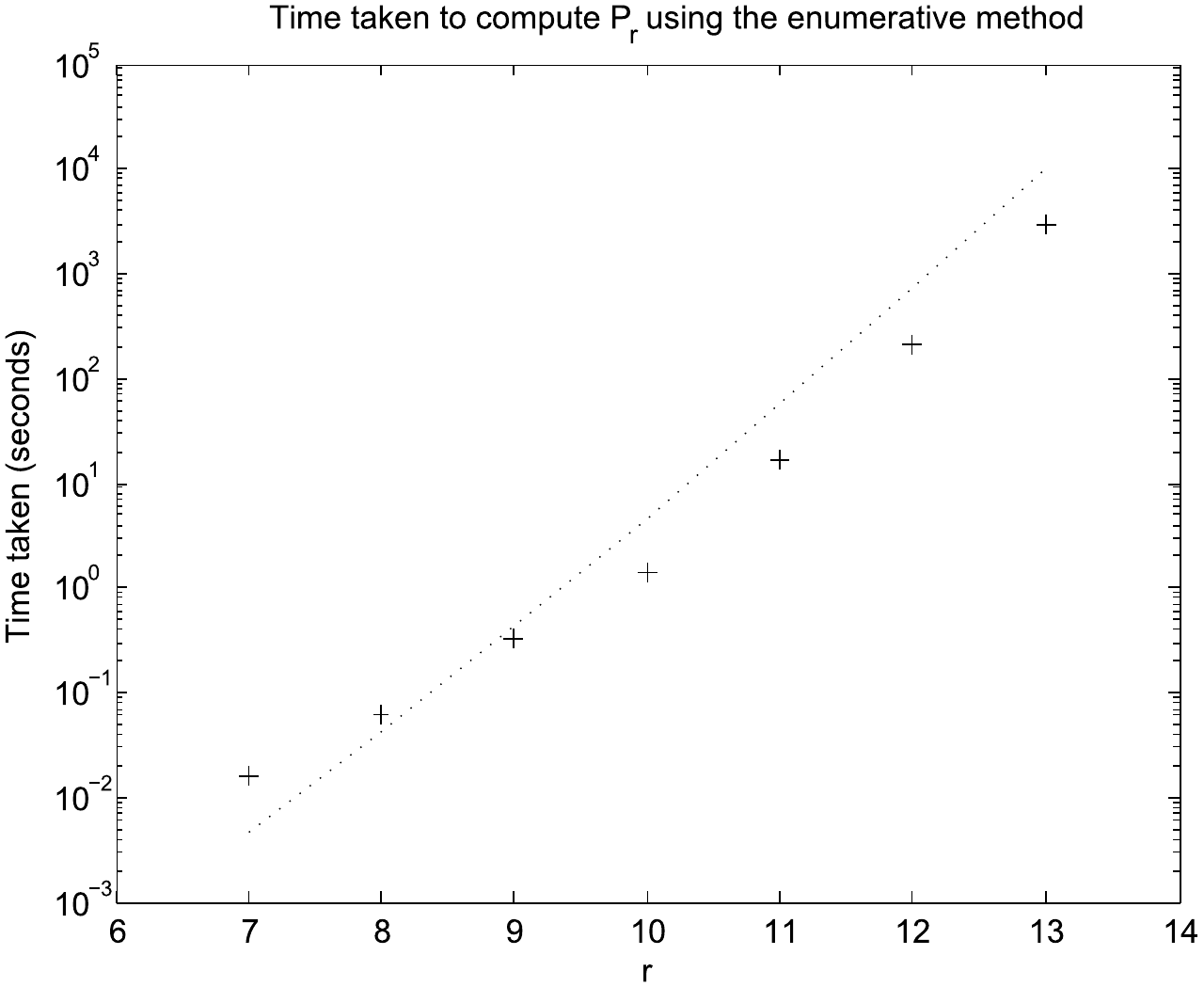}

}\hfill{}\subfloat[\label{fig:graph-b}]{\includegraphics[width=4cm]{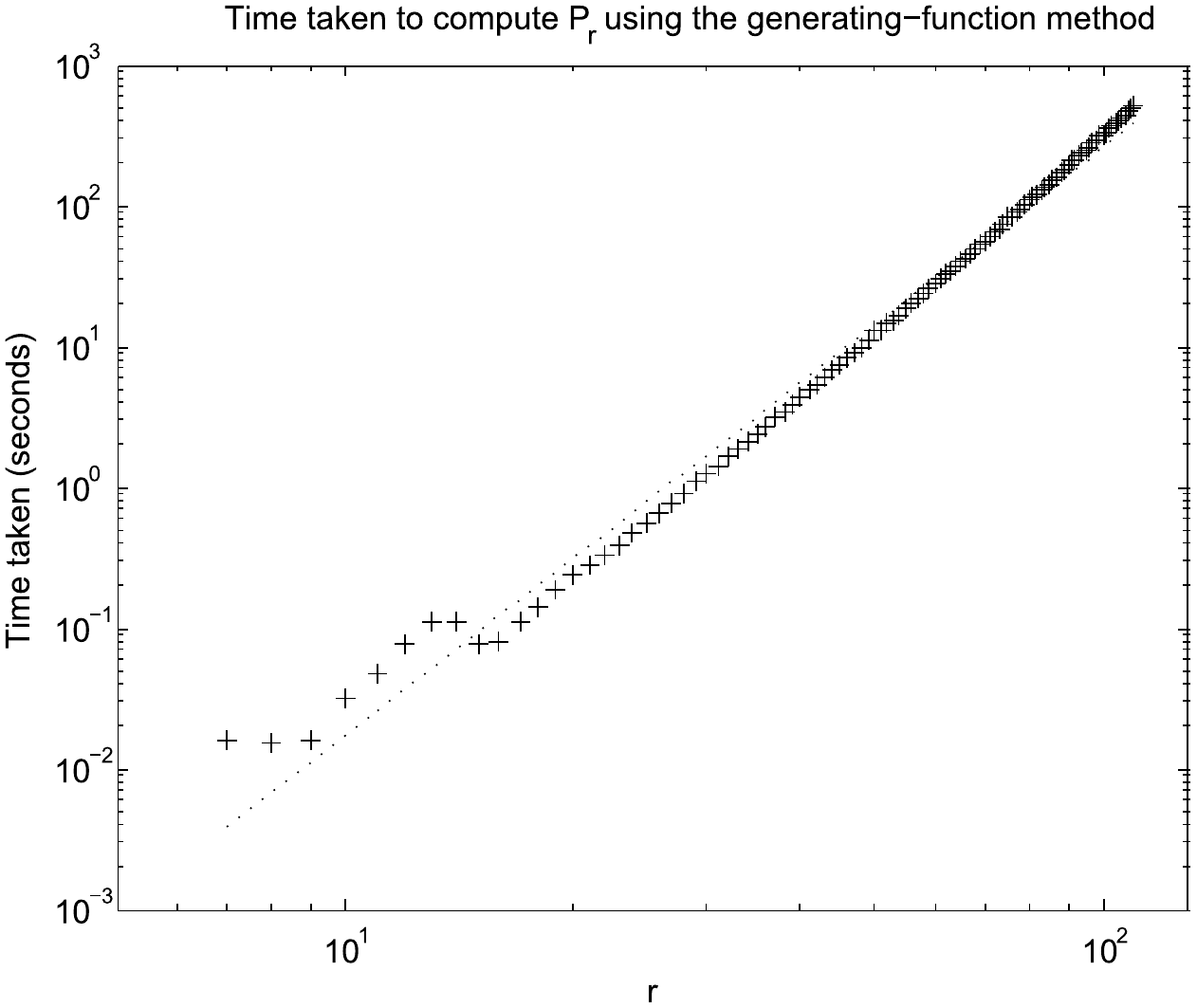}

}\hfill{}

\protect\caption{\label{fig:time-graph}Graphs of time taken to compute $P_{r}(m,n)$
for various values of $r$ using (a) the enumerative method and (b)
the generating-function method. The fitted curves are (a) the predicted
$\mathcal{O}(r\cdot r!)$ time dependence and (b) a power law fit
with exponent 4.167. Both methods were computed on a 2012 Dell XPS
12 with a 2.49GHz processor, and cases for small $r$ where the total
run-time was less than the 16ms resolution of the computer's clock
have been omitted.}
\end{figure}

In order to compare this generating function with existing results,
we calculated the coefficients of $P_{r}$ computationally using two
different methods:
\begin{enumerate}
\item \emph{Enumerative} -- enumerating permutations using Heap's algorithm
\cite{Heap1963a} and counting cycles using the scheme shown in figure
\ref{fig:contractions}.
\item \emph{Generating-function} -- expanding the generating function (\ref{eq:gf})
explicitly.
\end{enumerate}
The enumerative method has a predicted complexity of $\mathcal{O}(r\cdot r!)$,
which is reflected in the measured run-time of our C++ implementation
(see figure \ref{fig:graph-a}), in which all cases up to and including
$P_{13}$ (a total of $6\,749\,977\,113$ hypermaps) were computed
in just under an hour. This can be compared to previous enumerative
work by Walsh, in which he was able to compute $285\,764\,591\,114$
hypermaps (rooted, sensed and unsensed hypermaps with any number of
faces and up to $12$ darts) in a week \cite{Walsh2012}. This indicates
that our algorithm was able to process hypermaps at approximately
four times the rate that Walsh could; whether this improvement is
due to the algorithm or hardware is unclear.

We also implemented the generating-function method in C++, using a
specially designed class object to implement storage and manipulation
of integer polynomials in two variables, along with the open-source
MAPM library \cite{M.C.Ri} to handle the arbitrarily large integers
involved without loss of precision (calculating $P_{r}(m,n)$ requires
handling integers significantly larger than $r!$). This method appears
to run in polynomial time (based on measured run-times -- see figure
\ref{fig:graph-b}), a significant improvement in terms of speed over
the enumerative method. Indeed, the generating-function method reached
$P_{13}$ (as far as we took the enumerative method) in less than
a second, and in the hour it took the enumerative method to get that
far the generating-function method had reached $P_{90}$.

It must also be noted that the results given by these two methods
agreed with each other in all cases covered by both, and also agreed
with the values previously computed by Walsh \cite{Walsh2012}.

\subsection{Recursion relation\label{sub:Recursion-relation}}

We now have an explicit closed-form expression for our generating
functions $P_{r}(m,n)$, allowing us to compute any of them in isolation
for arbitrary $r$. However, we can go further and also derive a recursion
relation, allowing $P_{r}(m,n)$ to be computed quickly in terms of
the previous two generating functions. This arises from the fact that
$P_{r}(m,n)$ has the form of a hypergeometric sum, allowing us to
use Zeilberger's algorithm for deriving recursion relations \cite{Petkovsek1996}.

\bigskip{}

\begin{thm}
For any $r\ge1$,
\begin{equation}
(r+3)P_{r+2}(m,n)=(2r+3)(m+n)P_{r+1}(m,n)+r[(r+1)^{2}-(m-n)^{2}]P_{r}(m,n).\label{eq:P-Recurse}
\end{equation}
\end{thm}
\begin{proof}
First, we write
\[
P_{r}(m,n)=\sum_{k}F(r,k)
\]
where
\[
F(r,k)=\frac{(-1)^{k}}{r!}\binom{r-1}{k}\frac{\Gamma(m+r-k)}{\Gamma(m-k)}\frac{\Gamma(n+r-k)}{\Gamma(n-k)}
\]
and the sum is over all integers $k$ ($F$ is zero when $k<0$ or
$k\ge r$, so this is equivalent to (\ref{eq:gf})). The parameters
$m$ and $n$ have been omitted for simplicity. Using the EKHAD Maple
package \cite{DZeilberger}, we find that this satisfies
\begin{equation}
\begin{array}{c}
(r+3)F(r+2,k)-(2r+3)(m+n)F(r+1,k)+r[(m-n)^{2}-(r+1)^{2}]F(r,k)\\
=G(r,k+1)-G(r,k)\qquad\qquad\qquad
\end{array}\label{eq:F-recurse}
\end{equation}
where
\begin{eqnarray*}
G(r,k) & = & \frac{(-1)^{k}}{(r+2)!}\binom{r}{k-1}\frac{(m+r-k)!}{(m-k-1)!}\frac{(n+r-k)!}{(n-k-1)!}(k^{2}r-3kr^{2}-mnr+2r^{3}\\
 &  & +k^{2}+km+kn-7kr-3mn-mr-nr+7r^{2}-4k-m-n+8r+3).
\end{eqnarray*}
This can be verified by substitution and cancellation.

$G(r,k)$ is zero when $k<1$ or $k>r+1$, so when we sum (\ref{eq:F-recurse})
over all $k$, the right hand side telescopes to zero, and each $F(r,k)$
term in the left hand side sums to a $P_{r}(m,n)$ term, so
\[
(r+3)P_{r+2}(m,n)-(2r+3)(m+n)P_{r+1}(m,n)+r[(m-n)^{2}-(r+1)^{2}]P_{r}(m,n)=0,
\]
which is equivalent by rearrangement to (\ref{eq:P-Recurse}).
\end{proof}
\bigskip{}

Along with the initial cases $P_{1}(m,n)=mn$ and $P_{2}(m,n)=mn(m+n)$,
this gives us a simple method of computing any $P_{r}(m,n)$ recursively.

(\ref{eq:P-Recurse}) does not possess an obvious combinatorial interpretation,
although it does make a number of properties of $P_{r}(m,n)$ immediately
clear, i.e. that $P_{r}(m,n)$ is symmetric, and $P_{r}(m,m)$ is
a polynomial of order $r+1$ which is even when $r$ is odd and odd
when $r$ is even.

\subsection{Beyond one face\label{sub:Beyond-one-face}}

The method used in Theorem \ref{thm:closed-form} along with \cite{Dyer2014}
can be used more generally to compute other generating functions which
can be expressed in terms of the reduced density operator $\hat{\rho}^{A}$.
This fact can be used to compute generating functions for counting
rooted hypermaps with more than one face, and in this section we briefly
outline how this can be done using the two-face case as an example.

\bigskip{}

\begin{thm}
Let
\[
P_{a,b}(m,n)=\frac{\Gamma(mn+a+b)}{\Gamma(mn)}\langle\text{Tr}[(\hat{\rho}^{A})^{a}]\text{Tr}[(\hat{\rho}^{A})^{b}]\rangle
\]
for positive integers $r$ and $s$, and let $h_{r}^{(2)}(e,v)$ be
the number of rooted hypermaps with two faces, $e$ edges and $v$
vertices. Then
\[
\sum_{e,v}h_{r}^{(2)}(e,v)m^{e}n^{v}=\sum_{b=1}^{r-1}\frac{1}{b}\left(P_{r-b,b}(m,n)-P_{r-b}(m,n)P_{b}(m,n)\right).
\]
\end{thm}
\begin{proof}
\begin{figure}[h]
\hfill{}\includegraphics[width=10cm]{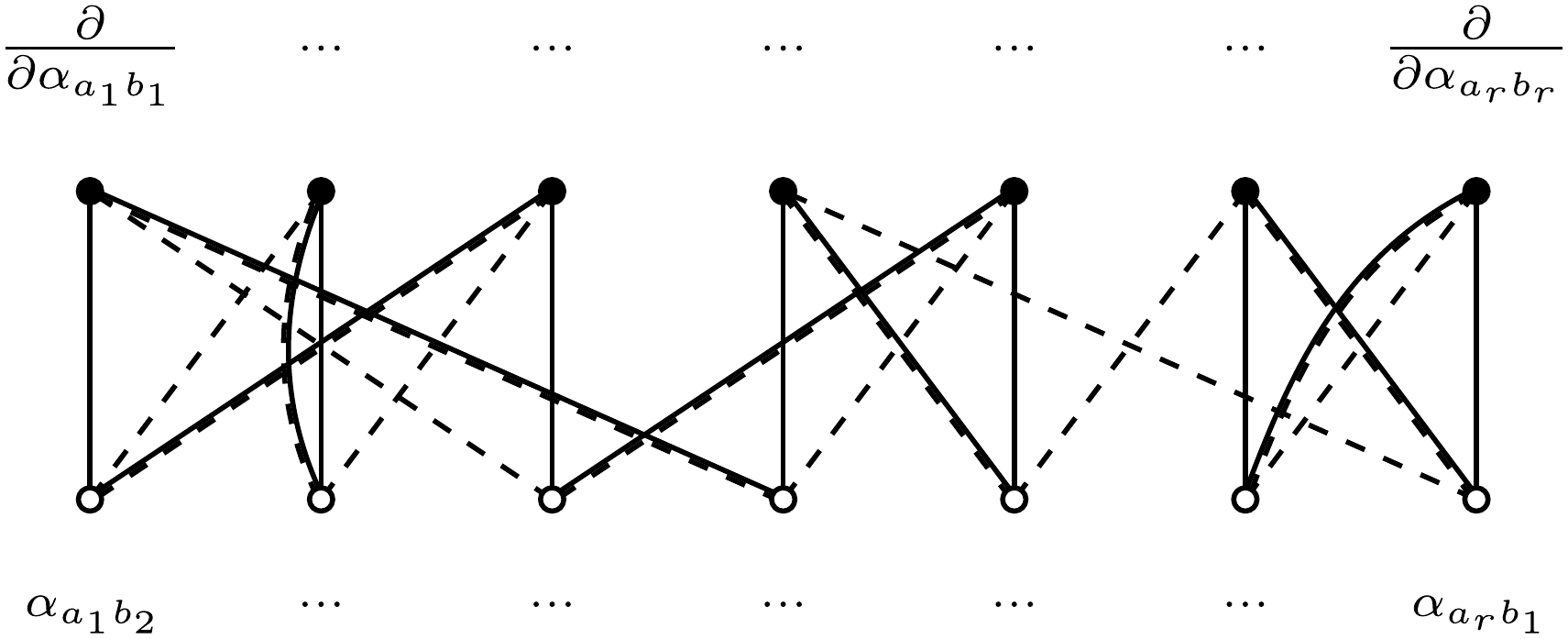}\hfill{}

\protect\caption{\label{fig:two-loop-1}A diagram representing a hypermap with two
face. This diagram shares most properties with figure \ref{fig:contractions},
except that the dashed lines have been rearranged so that there are
two loops of solid and dotted lines (of lengths three and four going
from left to right). This diagram in particular has been chosen because
it is connected; if the permutation associated with the double lines
did not join the two loops together the diagram (and resulting hypergraph)
would be disjoint, meaning it could not represent a hypermap.}
\end{figure}
Figure \ref{fig:contractions} demonstrates a representation of rooted
one-face hypermaps as diagrams with two distinct sets of $r$ vertices
and three distinct sets of $r$ edges connecting them. Two of these
edge types -- the solid and dashed edges -- combine to create a single
closed loop, which corresponds to the single face of the hypermap
it represents.

We assign to a two-face rooted hypermap a similar representation,
but with the solid and dashed edges forming two loops instead of one.
An example of such a diagram is shown in Figure \ref{fig:two-loop-1}.
If we take a diagram of this form and sum over all possible arrangements
of the double lines, then in the same way that diagrams with one loop
of length $r$ corresponds to 
\[
P_{r}(m,n)=\frac{\Gamma(mn+r)}{\Gamma(mn)}\langle\text{Tr}[(\hat{\rho}^{A})^{r}]\rangle,
\]
a diagram with two loops of length $a$ and $b$ corresponds to
\[
P_{a,b}(m,n)=\frac{\Gamma(mn+a+b)}{\Gamma(mn)}\langle\text{Tr}[(\hat{\rho}_{m}^{mn})^{a}]\text{Tr}[(\hat{\rho}_{m}^{mn})^{b}]\rangle.
\]
Proving this requires following the same method used in Theorem \ref{thm:matrix-integral},
and seeing that each trace term in the mean corresponds to a single
loop.

This function is not sufficient to give us the generating function
for two-face rooted hypermaps, however, as it over-counts for two
reasons. First, it counts cases where the diagram is disjoint. We
remove these cases by subtracting
\[
P_{a}(m,n)P_{b}(m,n),
\]
which generates the disjoint cases by treating the two loops as independent
one-loop diagrams placed next to each other.

Secondly, as rooting the hypermaps requires only fixing of one of
the two loops against cyclic permutation, there remains a cyclical
degeneracy in the second loop; if the diagram isn't disjoint then
cycling the second will necessarily produce a distinct diagram, so
each diagram is in an equivalence class of size $b$. We account for
this degeneracy by dividing by $b$.

Finally, to count all such hypermaps with $r$ darts, we sum over
all pairs of integers $(a,b)$ which sum to $r$, giving the generating
function
\begin{equation}
\sum_{e,v}h_{r}^{(2)}(e,v)m^{e}n^{v}=\sum_{b=1}^{r-1}\frac{1}{b}\left(P_{r-b,b}(m,n)-P_{r-b}(m,n)P_{b}(m,n)\right).\label{eq:2-face-gf}
\end{equation}

\end{proof}
\bigskip{}

This generating function can then be evaluated exactly using the same
tools we applied to the one-face case. But without going that far
we can already derive the following result, a direct analogue of Corollary
\ref{cor:all-count}:

\bigskip{}

\begin{cor}
There are
\[
\sum_{b=1}^{r-1}\frac{r!-b!(r-b)!}{b}
\]
rooted hypermaps with two faces and $r$ darts.\end{cor}
\begin{proof}
(\ref{eq:2-face-gf}) gives us a generating function for counting
rooted hypermaps with two faces and $r$ darts. In order to sum over
all possible numbers of edges and vertices we simply set $m$ and
$n$ to unity in this expression.

Due to the fact that $\hat{\rho}^{A}$ has unit trace,
\[
P_{r}(1,1)=\frac{\Gamma(1+r)}{\Gamma(1)}\langle1\rangle=r!
\]
and
\[
P_{a,b}(1,1)=(a+b)!,
\]
so
\begin{eqnarray*}
\sum_{e,v}h_{r}^{(2)}(e,v) & = & \sum_{b=1}^{r-1}\frac{P_{r-b}(1,1)-P_{b}(1)P_{r-b}(1)}{b}\\
 & = & \sum_{b=1}^{r-1}\frac{r!-b!(r-b)!}{b}.
\end{eqnarray*}

\end{proof}

\section{Discussion}

We have seen that our method for finding closed-form generating functions
has a quantitative advantage over the direct enumerative method in
terms of the time it takes to compute the actual numbers of hypermaps.
However, this is not the only benefit of the method.

Generating functions are a very powerful tool in enumerative combinatorics,
as they allow numerous operations acting on the objects being counted
to be represented by simple algebraic operations. For instance, the
union of two sets is enumerated using the sum of their generating
functions, and the product of two sets is enumerated by the product
of their generating functions\cite[pp 3-4]{Stanley1997}.

The general methods used in this paper have potentially very wide
applicability. All of the examples covered in this paper are possible
because matrix polynomials integrated over the unit sphere have a
connection to the cycle representation of permutations, as we showed
in Sections \ref{sub:Permutations} and \ref{sub:Hypermaps}. Many
different combinatoric quantities can be expressed in terms of permutations
(the 3-constellation representation of hypermaps provides just one
example), so this method may be usable to construct generating functions
in other cases where objects can be represented using permutations.

The fact that the matrix integrals used to compute the generating
functions for one-face rooted hypermaps have a relation to bipartite
quantum systems is also of interest. That there is a link between
\emph{bipartite} quantum systems and \emph{bipartite} maps (equivalent
to hypermaps) in particular is clearly not a coincidence; The parameters
$m$ and $n$ in the function $P_{r}(m,n)$ are identified with the
dimensions of the two subsystems of the bipartite quantum system,
and are then also linked to the edges and vertices of the hypermap
respectively (equivalently the two colours of vertices in the bipartite
bicoloured map) when $P_{r}$ is used as a generating function, so
there is a clear correspondence between the bipartite natures of the
quantum systems and the hypermaps.

The relation to quantum systems has had two immediate benefits here.
First, the notation using means of functions of the reduced density
matrix provides a useful shorthand for expressing the bulky matrix
integrals and multiderivatives (see (\ref{eq:matrix-integral}), for
example), which is even extendable to multiple faces as seen in Section
\ref{sub:Beyond-one-face}. Second, the use of expressions in terms
of the reduced density matrix led directly to a method of evaluating
the generating functions in closed form.

\section{Acknowledgements}

The work in this paper was supported by an EPSRC research studentship
at the University of York.

We thank Bernard Kay for commenting on an earlier version of this
paper. We also thank Christopher Hughes for for introducing us to
Zeilberger's algorithm.

\bibliographystyle{plain}
\bibliography{Hypermaps1}

\end{document}